\documentclass{article}
\usepackage{amsmath}
\usepackage{amsfonts}
\usepackage{amssymb}
\usepackage{amsthm}
\usepackage[a4paper]{geometry}

\newcommand\blfootnote[1]{%
  \begingroup
  \renewcommand\thefootnote{}\footnote{#1}%
  \addtocounter{footnote}{-1}%
  \endgroup
}

\newtheorem{theorem}{Theorem}

\newtheorem{lemma}[theorem]{Lemma}

\newtheorem{proposition}[theorem]{Proposition}

\newtheorem{remark}[theorem]{Remark}

\theoremstyle{definition}

\newtheorem{example}[theorem]{Example}

\newcommand{\FF}{\mathcal{F}}

\newcommand{\F}{\mathbb{F}}

\newcommand{\Q}{\mathbb{Q}}

\begin{document}
\title{Good Towers of Function Fields}

\author{Alp Bassa, Peter Beelen and Nhut Nguyen}
\date{}
\maketitle

\abstract{In this paper, we will give an overview of known and new techniques on how one can obtain explicit equations for candidates of good towers of function fields. The techniques are founded in modular theory (both the classical modular theory and the Drinfeld modular theory). In the classical modular setup, optimal towers can be obtained, while in the Drinfeld modular setup, good towers over any non-prime field may be found. We illustrate the theory with several examples, thus explaining some known towers as well as giving new examples of good explicitly defined towers of function fields.}

\blfootnote{Alp Bassa is supported by T\"{u}bitak Proj.~No.~112T233. Peter Beelen and Nhut Nguyen gratefully acknowledge the support from the Danish National Research Foundation and the National Science Foundation of China (Grant No.~11061130539) for the Danish-Chinese Center for Applications of Algebraic Geometry in Coding Theory and Cryptography.}

\blfootnote{Final version of this manuscript will appear in \emph{Algebraic Curves and Finite Fields: Codes, Cryptography, and other Emergent
Applications (H. Niederreiter, A. Ostafe, D. Panario, and A. Winterhof, eds.),
de Gruyter, Berlin}.}

\section{Introduction}

The question of how many rational points a curve of genus $g$ defined
over a finite field $\mathbb F_q$ can have, has been a central and
important one in number theory. One of the landmark results in the
theory of curves defined over finite fields was the theorem of Hasse
and Weil, which is the congruence function field analogue of the
Riemann hypothesis. As an immediate consequence of this theorem one
obtains an upper bound for the number of rational points on such a
curve in terms of its genus and the cardinality of the finite field.
It was noticed however by Ihara \cite{ihara} and Manin \cite{manin1}
that this bound can be improved for large genus and the asymptotic
study over a fixed finite field was then initiated by Ihara. An
asymptotic upper bound on the number of rational points was given by
Drinfeld and Vladut \cite{dv}. More precisely they showed that if $(C_i)_i$ is a family of curves all defined over $\F_q$ such that $C_i$ has genus $g_i$ (tending to infinity as $i$ tends to infinity) and $N_i$ rational points, then $$\limsup_{i\to \infty} \frac{N_i}{g_i} \le \sqrt{q}-1.$$

Finding curves of large genera with many points is a difficult task
and there have basically been three approaches: class field theory
(see among others \cite{xingnied,serre}), explicit constructions (see
among others \cite{elkies,GSINV,GSJNT,tame}) and reductions of modular
curves of various types (see among others
\cite{ihara,ge2,TVZ,manin}). With these techniques it is possible
to construct sequences of curves having many points compared to their
genera asymptotically and in some cases even attaining the Drinfeld--Vladut bound, in which case the sequence of curves is called optimal.

In \cite{GSJNT}, Garcia and Stichtenoth introduced the following
optimal sequence of function fields $(F_n)_{n \ge 0}$ over $\mathbb F_
\ell$, where $\ell=q^2$ : Let $F_0=\mathbb F_\ell(x_0)$  and define
$F_{n+1}=F_n(x_{n+1})$ where
$$x_{n+1}^q+x_{n+1}=\frac{x_n^q}{x_n^{q-1}+1},$$
for $n\geq 0$.
Because of its recursive behaviour, we say that the tower is recursive, satisfying the recursive equation
\begin{equation}\label{eq:jnteq}
y^q+y=\frac{x^q}{x^{q-1}+1}.
\end{equation}
In \cite{elkies,elkiesd}, Elkies gave a modular interpretation for
this and for all other known optimal recursive towers. More precisely
he showed that all known examples of tame, (respectively wild) optimal
recursive towers correspond to reductions of classical (respectively
Drinfeld) modular curves. Moreover, he found several other equations
for such towers, by studying  reductions of Drinfeld-, elliptic- and
Shimura-modular curves very explicitly and gave an explanation for the
recursive nature of these towers. Until now many explicitly known, recursively defined towers have a modular explanation.
As an example of this phenomenon, we give a modular interpretation for a good recursive tower given in \cite{loetter}.

Elkies showed that the reduction of the tower of Drinfeld modular curves $(X_0(T^n))_{n\geq
2}$ at the prime $T-1$ is a recursive tower satisfying the recursive equation
\begin{equation}\label{eq:elkieseq}
(y+1)^{q-1}\cdot y=\frac{x^q}{(x+1)^{q-1}}.
\end{equation}
This is an optimal tower, which was also studied in detail in
\cite{bezerragarcia}. It is a subtower of the tower defined by (\ref{eq:jnteq}).
In this paper we elaborate further on the ideas of Elkies. Note that the recursive equation in Equation \eqref{eq:elkieseq} has depth one. With this we mean that the variable $x_{n+1}$ in the $(n+1)$-th step of the tower is related to only the previous variables $x_n$ by the recursive equation.

We show how the defining equations for these modular towers can be read off
directly from the modular polynomial, and how this, in general,
leads to recursions of depth 2. More precisely, we show that the tower can be defined by recursive equations which relate in the $(n+1)$-th step of the tower (for $n\ge 1$), the variable $x_{n+1}$ to both $x_n$ and $x_{n-1}$. With this approach, finding explicit recursive towers turns out to be an easy task, once the corresponding modular polynomials are known. To illustrate this, we work out the equations
for a few cases of Drinfeld modular towers.

In the above Drinfeld modular theory was considered over the polynomial ring $\F_q[T]$. In the last section of the paper, we study a variation where this ring is replaced by the coordinate ring of an elliptic curve. We illustrate the ideas by going through a specific example in detail.

\section{The Drinfeld modular towers $(X_0(P^n))_{n\ge0}$}\label{section:two}

In this section we will restrict ourselves to the case of Drinfeld modular curves. However, the classical case of elliptic modular curves is analogous. Therefore we will on occasion state some observation for the classical case also. For more information on Drinfeld modules, the reader is referred to \cite{rosen,goss}. For more information on Drinfeld modular curves, see for example \cite{gekeler}. We denote by $\F$ the field $\F_q(T)$ and let $N \in \F_q[T]$ be a monic polynomial. The field $\F$ will play the role of constant field in the towers we find. From these, towers with a finite field as a constant field can be obtained by reducing the defining equations by a suitably chosen prime element $L$ of $\F_q[T].$ More precisely the constant field of such a reduced tower is $\F_L:=\F_q[T]/(L).$ To describe how to obtain (unreduced) towers, we will use the language of Drinfeld modules.

Let $\phi$ be a Drinfeld module of rank two with $j$-invariant $j_0$ and $\phi'$ be an $N$-isogenous Drinfeld module with $j$-invariant $j_1$. The Drinfeld modular polynomial $\Phi_N(X,Y)$ relates these $j$-invariants, more precisely it holds that $\Phi_N(j_0,j_1)=0$. Thinking of $j_0$ as a transcendental element, we can use this equation to define a so-called Drinfeld modular curve $X_0(N)$. If we want to emphasize the role of $N$, we will write $j_1=j_1(N)$. It should be noted that $j_0$ is independent of $N$, but it will be convenient to define $j_0(N):=j_0$. The function field $\F(X_0(N))$ of $X_0(N)$ is therefore given by $\F(j_0(N),j_1(N))$. Moreover, it is known, see \cite{bae}, that

\begin{equation}\label{eq:extdegree}
[\F(j_0(N),j_1(N)):\F(j_0(N))]=q^{\deg(N)}\prod_{\substack{P | N \\\ P \ \mathrm{ prime}}}\left(1+\frac{1}{q^{\deg(P)}}\right).
\end{equation}

In principle the work of finding an explicit description of the function field $\F(X_0(N))$ is done, once the modular polynomial $\Phi_N(X,Y)$ has been computed. However, for general $q$ the Drinfeld modular polynomial is not known explicitly. Even in the case $N=T$ it has only been determined recently \cite{bassabeelendrinfeld}. For a given $q$ it can be computed, but this is not always an easy task, since the coefficients of this polynomial tend to get very complicated as the degree of the polynomial $N$ increases. However, following Elkies's ideas (\cite{elkies,elkiesd}) from the modular polynomial $\Phi_P(X,Y)$ for a fixed polynomial $P$, the function fields of the Drinfeld modular curves $X_0(P^n)$ can be described easily in an explicit way. The reason for this is that for polynomials $P,Q \in \F_q[T]$ a $PQ$-isogeny can be written as the composite of a $P$-isogeny and a $Q$-isogeny, which implies that there is a natural projection from $X_0(PQ)$ to $X_0(P)$ or equivalently an inclusion of function fields $\F(X_0(P)) \subset \F(X_0(PQ))$. This implies that the function field $\F(X_0(P^n))$ also contains the function fields $\F(X_0(P^e))$, for any integer satisfying $1\le e \le n$, and hence $j_1(P^e) \in \F(X_0(P^n))$. Defining $j_e(P):=j_1(P^e)$ for $e \ge 1$, we see that $j_e(P) \in \F(X_0(P^n))$ for $1 \le e \le n$. Since $j_0$ is independent of $P$, we also have $j_0(P)=j_0(P^n) \in \F(X_0(P^n))$. Therefore the field $\F(X_0(P^n))$, is the composite of the fields $\F(j_e(P),j_{e+1}(P))$ for $e=0,\dots,n-1$. Since $P^{e+1}=PP^e$, any $P^{e+1}$-isogeny can be written as the composite of a $P$-isogeny and a $P^e$-isogeny. This means that $j_e(P)$ and $j_{e+1}(P)$ correspond to $P$-isogenous Drinfeld modules and hence we have $\Phi_P(j_e(P),j_{e+1}(P))=0$ for any $e$ between $0$ and $n-1$. We see that $\F(X_0(P^n))$ is the composite of $n$ fields isomorphic to $\F(X_0(P))=\F(j_0(P),j_1(P))$, the function field of $X_0(P)$. This observation led Elkies to construct a number of recursively defined {\emph{towers}} $(X_0(P^n))_{n\ge 2}$ of modular curves in \cite{elkies,elkiesd}. In \cite{elkies} several models defined over $\Q$ of classical modular curves are given, while in \cite{elkiesd} the reduction mod $T-1$ of the Drinfeld modular tower $X_0(T^n)_{n\ge 2}$ was described.

We consider the function field of $X_0(P^n)$. We have $$\F(X_0(P^n))=\F(j_0(P),j_1(P),\dots,j_{n-1}(P),j_n(P)).$$ So we can think of $\F(X_0(P^n))$ as iteratively obtained from $\F(j_0(P))$ by adjoining the elements $j_1(P),j_2(P),\dots,j_n(P)$, where $j_{e+1}(P)$ is a root of the polynomial $\Phi_P(j_{e}(P),t) \in \F(X_0(P^e))[t]$ for $0 \le e < n$. However, except for $j_1(P)$ these polynomials are not irreducible. In fact the extension $\F(X_0(P^2))/\F(X_0(P))$ has degree $q^{\deg P}$ by Equation (\ref{eq:extdegree}). This means that the polynomial $\Phi_P(j_1(P),t) \in \F(j_0(P),j_1(P))[t]$ has a factor $\Psi_P(j_0(P),j_1(P),t)$ of degree $q^{\deg P}$ such that $$\Psi_P(j_0(P),j_1(P),j_2(P))=0.$$ We can assume that $\Psi_P(j_0(P),j_1(P),t) \in \F[j_0(P),j_1(P)][t]$, by clearing denominators if necessary. Then  the trivariate polynomial $\Psi_P(X,Y,Z) \in \F[X,Y,Z]$  satisfies $$\Psi_P(j_{e-1}(P),j_e(P),j_{e+1}(P))=0$$ for all $0<e<n$. The function field $\F(X_0(P^n))$ can hence be generated recursively by the equations $\Phi_P(j_0(P),j_1(P))=0$ and $\Psi_P(j_{e-1}(P),j_e(P),j_{e+1}(P))=0$ for $0<e<n$. Note that the depth of the recursion is two in general, meaning that to obtain the minimal polynomial of $j_{e+1}(P)$ over $\F(j_0(P),\dots,j_{e}(P))$ for $e\ge 1$, we need both $j_e(P)$ and $j_{e-1}(P)$. We arrive at the following proposition.

\begin{proposition}\label{prop:depth two}
Let $P \in \F_q[T]$ be a polynomial and $n \ge 0$ an integer. The function field $G_n$ of the Drinfeld modular curve $X_0(P^n)$ is generated by elements $j_0,\dots,j_n$ satisfying:
$$\Phi_P(j_0,j_1)=0,$$ with $\Phi_P(X,Y)$ the Drinfeld modular polynomial corresponding to $P$ and
$$\Psi_P(j_{e-1},j_e,j_{e+1})=0, \makebox{ for } 1\le e <n,$$ with $\Psi_P(X,Y,Z)$ a suitable trivariate polynomial of $Z$-degree $q^{\deg P}$.
Consequently, the tower of function fields $\mathcal{G}:=(G_n)_{n \ge 0}$ can be recursively defined by a recursion of depth two in the following way:
$$G_0:=\F(j_0),$$ $$G_1:=\F(j_0,j_1), \makebox{ where } \Phi_P(j_0,j_1)=0$$ and for $n\ge 1$ $$G_{n+1}:=G_{n}(j_{n+1}) \makebox{ where } \Psi_P(j_{n-1},j_n,j_{n+1})=0.$$
\end{proposition}

\begin{remark}\label{rem:casePprime}
The polynomial $\Psi_P(X,Y,Z)$ is easy to describe if $P$ is a prime. In that case $\deg_Y(\Phi_P(X,Y))=q^{\deg P}+1$. Since $\Phi_P(X,Y)$ is a symmetric polynomial, it holds that $$\Phi_P(j_1(P),j_0(P))=\Phi_P(j_0(P),j_1(P))=0.$$ Therefore, the polynomial $\Phi_P(j_1(P),t) \in \F(X_0(P))[t]$ has the factor $t-j_0(P)$. The factor $\Psi(j_0(P),j_1(P),t)$ can be obtained by dividing $\Phi_P(j_1(P),t)$ by $t-j_0(P)$. Note that in this case automatically $\deg_t \Psi_P(j_0(P),j_1(P),t)=q^{\deg P}$ and $$\Psi_P(j_0(P),j_1(P),j_2(P))=0,$$ as desired. A similar remark holds for the classical case: if $p$ is a prime number, then the classical modular polynomial $\Phi_p(X,Y)$ is a symmetric polynomial having degree $p+1$ in both $X$ and $Y$. The polynomial $\Phi_p(j_1(p),t) \in \Q(j_0(p),j_1(p))[t]$ has a factor of degree one in $t$ (namely $t-j_0(p)$) and a factor of degree $p$.
\end{remark}

By \cite{schweizer3} $X_0(P)$ is rational if and only if $P$ has degree one or two. In that case the tower $(\F(X_0(P^n)))_{n\ge 1}$ can be generated in a simpler way. Let $e\ge 1$ and let $u_{e-1}(P)$ be a generating element of $\F(j_{e-1}(P),j_e(P))$ over $\F$. Then $j_{e-1}(P)=\psi(u_{e-1}(P))$ and $j_e(P)=\phi(u_{e-1}(P))$ for certain rational functions $\psi(t)=\psi_0(t)/\psi_1(t)$ and $\phi(t)=\phi_0(t)/\phi_1(t)$. Here $\psi_0(t)$ and $\psi_1(t)$ (resp. $\phi_0(t)$ and $\psi_1(t)$) denote relatively prime polynomials. Since $\F(u_{e-1}(P))=\F(j_{e-1}(P),j_e(P))$, one can generate the function field of $X_0(P^n)$ for $n \ge 1$ by $u_0(P),\dots,u_{n-1}(P)$. These generating elements satisfy the equations $\psi(u_e(P))=\phi(u_{e-1}(P))$ with $1\le e <n$, since $\psi(u_e(P))=j_e(P)=\phi(u_{e-1}(P))$. Similarly as before, one can find generating relations of minimal degree by taking a factor $f_P(u_0(P),t)$ of $\psi_0(t)\phi_1(u_{0}(P))-\psi_1(t)\phi_0(u_{0}(P))$ of degree $q^{\deg P}$ such that $f(u_0(P),u_1(P))=0$. The function field $\F(X_0(P^n))$ with $n\ge 1$ can then recursively be defined by the equations $f(u_{e-1},u_e)=0$ for $1\le e <n$. We arrive at the following proposition.

\begin{proposition}\label{prop:depth one}
Let $P \in \F_q[T]$ be a polynomial of degree one or two and $n \ge 0$ an integer. There exists a bivariate polynomial $f_P(X,Y) \in \F[X,Y]$ of $Y$-degree $q^{\deg P}$ such that the function field $G_n$ of the Drinfeld modular curve $X_0(P^n)$ is generated by elements $u_0,\dots,u_{n-1}$ satisfying:
$$f_P(u_{e-1},u_e)=0, \makebox{ for } 1\le e < n.$$
Consequently, the tower of function fields $\mathcal{G}:=(G_n)_{n \ge 1}$ can be defined by a recursion of depth one:
$$G_1:=\F(u_0)$$ and for $n\ge 1$ $$G_{n+1}=G_{n}(u_{n+1}) \makebox{ where } f_P(u_n,u_{n+1})=0.$$
\end{proposition}

Finally, if $P$ is a polynomial of degree one, then both $X_0(P)$ and $X_0(P^2)$ are rational. In that case, there exist $u_{e-1}(P),u_e(P)$ as above and $v_{e-1}(P)$ such that $\F(u_{e-1}(P),u_{e}(P))=\F(v_{e-1}(P))$ for $e>0$. Similarly as above, there exist rational functions $\psi'(t)$ and $\phi'(t)$ such that
$u_{e-1}(P)=\psi'(v_{e-1}(P))$ and $u_e(P)=\phi'(v_{e-1}(P))$. These rational functions have degree $q^{\deg P}=q$, since $[\F(v_{e-1}(P)):\F(u_{e-1}(P))]=[\F(v_{e-1}(P)):\F(u_{e}(P))]=q.$ The function field $\F(X_0(P^n))$ with $n\ge 2$ can then recursively be defined by the equations $\psi'(v_e(P))=\phi'(v_{e-1}(P))$ for $1\le e <n-1$. The depth of the recursion is one (since the defining equation relates $v_e(P)$ to $v_{e-1}(P)$ only) and moreover, the variables can be separated in the defining equations. Since we assume $\deg P=1$, this puts a heavy restriction on the number of possibilities. In fact, without loss of generality we may assume that $P=T$. In the next section we will describe this case in detail, obtaining explicit equations describing the Drinfeld modular tower $\F(X_0(T^n))_{n\ge 2}$. In the case of classical modular curves, Elkies in \cite{elkies} gave, among others, several similar examples by considering (prime) numbers $p$ such that the genus of the classical modular curves $X_0(p)$ and $X_0(p^2)$ is zero. This is the case for $p \in \{2,3,5\}$.

The towers $(\F(X_0(P^n)))_{n\ge 0}$ are also useful for obtaining interesting towers with finite constant fields, since Gekeler showed the following:

\begin{theorem}[\cite{ge2}]\label{thm:optimal}
Given a prime $L \in \F_q[T]$, denote by $\F_L$ the finite field $\F_q[T]/(L)$. Moreover, write $\F_L^{(2)}$ for the quadratic extension of $\F_L$.  The reduction modulo any prime $L \in \F_q[T]$ not dividing $P$ of the tower $(X_0(P^n))_{n\ge 0}$ gives rise to an asymptotically optimal tower over the constant field $\F_L^{(2)}$.
\end{theorem}

The above theorem implies that the tower found in \cite{elkiesd}, being the reduction of $(X_0(T^n))_{n\ge 0}$ modulo $T-1$, is asymptotically optimal over the constant field $\F_{T-1}^{(2)}=\mathbb{F}_{q^2}.$ Now we will give several examples. Sometimes we do not give all details, since this would fill many pages. Several computations were carried out using the computer algebra package MAGMA \cite{magma}. For example all Drinfeld modular polynomials below were calculated using MAGMA. On occasion, we will perform all calculations sketched above for a reduced version of the tower $(\F(X_0(P^n)))_{n\ge 0}$, since the resulting formulas are usually much more compact after reduction. In all examples in this section, it is assumed that $q=2$, while $P$ will be a polynomial of degree one or two.

\begin{example}[$P=T, q=2$]
By \cite{schweizer}, the Drinfeld modular polynomial of level $T$ in case $q=2$ is given by
\begin{equation*}
\begin{split}
\Phi_{T}(X,Y)  & = X^3+Y^3+T(T+1)^3(X^2+Y^2)+T^2(T+1)^6(X+Y)\\
 &+T^3(T+1)^9+X^2Y^2+(T+1)^3(T^2+T+1)XY+T(X^2Y+XY^2).
\end{split}
\end{equation*}
The polynomial $\Psi_T(X,Y,Z)$ can readily be found using Remark \ref{rem:casePprime}:
\begin{equation*}
\begin{split}
\Psi_{T}(X,Y,Z)
& =  Z^2 + (X + (Y^2 + TY + T(T+1)^3))Z + X^2\\
& + (Y^2 + TY + T(T+1)^3)X + TY^2\\
& + (T^2+T+1)(T+1)^3Y + T^2(T+1)^6\\
\end{split}
\end{equation*}
Using Proposition \ref{prop:depth two}, we can in principle now describe the tower of function fields of the modular curves $(X_0(T^n))_{n \ge 0}$. However, we can use Proposition \ref{prop:depth one} to find a recursive description of depth one. First we need a uniformizing element $u_0$ of $\F(j_0,j_1)$.
Using a computer, one finds
$$u_0=\frac{T^3(T^2j_0+T^2+T^4+T^6+1+Tj_1+T^2j_1+Tj_0+j_0j_1)}{(T^3+j_1^2+T^2+j_0+Tj_1+T^3j_0+T^7+T^4j_1+T^6}.$$
Expressing $j_0$ and $j_1$ turns out to give a more compact formula.

$$j_0=\dfrac{(u_0+T)^3}{u_0} \makebox{ and } j_1=\dfrac{(u_0+T^2)^3}{u_0^2}.$$

This means that the variables $u_0$ and $u_1$ satisfy the equation:
$$\dfrac{(u_0+T^2)^3}{u_0^2}=\dfrac{(u_1+T)^3}{u_1}.$$
However, this is not an equation of minimal degree. As explained before Proposition \ref{prop:depth one}, we can find an equation of degree (in this case) two by factoring:
$$(X+T^2)^3Y+(Y+T)^3X^2=(XY+T^3)(X^2+XY^2+XYT+YT^3).$$ We find that $f_T(X,Y)=X^2+XY^2+XYT+YT^3$. This polynomial recursively defines the tower of function fields of the modular curves $(X_0(T^n))_{n \ge 1}$ as in Proposition \ref{prop:depth one}.
\end{example}

\begin{example}[$P=T^2+T+1, q=2$]
The Drinfeld modular polynomial of level $T^2+T+1$ is given by
\begin{equation*}
\footnotesize
\begin{split}
\Phi_{T^2+T+1}(X,Y)  & = X^5+Y^5 + X^4Y^4 + (T^2 + T + 1)(X^4Y^2+X^2Y^4)\\
& + (T^2 + T + 1)(X^4Y+XY^4)\\
& + T^3(T+1)^3(T^2+T+1)(X^4+Y^4)  \\
& + T^2(T+1)^2(T^2+ T + 1)X^3Y^3 \\
& + (T^2+T)(T^2+T+1)(T^3+T+1)(T^3+T^2+1)(X^3Y^2+X^2Y^3) \\
& + T^3(T+1)^3(T^2+T+1)(X^3Y+XY^3)\\
& + T^6(T+1)^6(T^2+T+1)^2(X^3+Y^3)\\
& + T^5(T+1)^5(T^2+T+1)(T^4+T+1)X^2Y^2\\
& + T^6(T+1)^6(T^2+T+1)(T^4+T+1)(X^2Y+XY^2)\\
& + T^9(T+1)^9(T^2+T+1)^3(X^2+Y^2) + T^{11}(T+1)^{11}X Y
\end{split}
\end{equation*}
As in the previous example one can use Remark \ref{rem:casePprime}, to find the trivariate polynomial $\Psi_{T^2+T+1}(X,Y,Z)$. Finding a uniformizing element $u_0$ of $\F(X_0(T^2+T+1))$ is somewhat more elaborate. Since such a uniformizing element fills several pages, it is omitted. Below we will state the reduction of $u_0$ modulo $T$ and $T+1$, so the reader can get an impression of its form. Once $u_0$ is found, $j_0$ and $j_1$ can be expressed in terms of it. In this case we find:
$$j_0=\dfrac{(u_0+1)^3(u_0^2+u_0+T^2+T+1)}{u_0}$$
and
$$j_1=\dfrac{(u_0+T^2+T+1)^3(u_0^2+u_0+T^2+T+1)}{u_0^4}$$
To find the polynomial $f_{T^2+T+1}(X,Y)$, we need to factor the polynomial

\begin{equation*}
\small
\begin{split}(Y^5+(T^2+T+1)Y^3+(T^2+T+1)Y^2+(T^2+T)Y+(T^2+T+1))X^4 + & \\ Y(X^5+(T^2+T)X^4+(T^2+T+1)^2X^3+(T^2+T+1)^3X^2+(T^2+T+1)^4),
\end{split}
\end{equation*}
whose factors are $XY+T^2+T+1$ and
\begin{equation*}
\begin{split}
f_{T^2+T+1}(X,Y)& =Y^4X^3 + (T^2 + T + 1)(Y^3X^2 + Y^2X^3 + (T^2 + T + 1)Y^2X\\ & + YX^3 + (T^2 + T + 1)YX^2 + (T^2 + T + 1)^2Y) + X^4
\end{split}
\end{equation*}
The polynomial $f_{T^2+T+1}(X,Y)$ recursively defines the tower of function fields of the modular curves $(X_0((T^2+T+1)^n))_{n \ge 1}$ as in Proposition \ref{prop:depth one}.

We consider the reduction modulo $T$ or $T+1$ of this tower, which by Theorem \ref{thm:optimal} gives an optimal tower over $\mathbb{F}_4$. While a uniformizing element of $\F(X_0(T^2+T+1))$ was too long to be stated, over $\mathbb{F}_4(X_0(T^2+T+1))$ it is given by
$$
u_0:=\dfrac{j_0^4j_1^3 + j_0^4j_1^2 + j_0^4j_1 + j_0^4 + j_0^3j_1^7 + j_0^3j_1^6 + j_0^3j_1^4 + j_0^2j_1^5
    + j_0 j_1^5 + j_0 j_1^4 + j_1^6 + j_1^4}{j_1^8}
$$
Reducing the above found polynomial $f_{T^2+T+1}(X,Y)$ modulo $T$ or $T+1$, we now explicitly find that the polynomial
$$Y^4X^3 + Y^3X^2 + Y^2X^3 + Y^2X + YX^3+ YX^2 + Y + X^4$$
recursively defines an optimal tower over $\F_4$.
\end{example}

\begin{example}[$P=T^2+T, q=2$]
In the previous examples, the polynomial $P$ was a prime, but in this example we will consider the composite polynomial $P=T^2+T$. The Drinfeld modular polynomial of level $T^2+T$ has $Y$-degree $9$ by Equation \ref{eq:extdegree}. Using a computer, one finds:
\begin{equation*}
\footnotesize
\begin{split}
\Phi_{T^2+T}(X,Y)
& =  X^9 +Y^9+ (X^8Y^4+X^4Y^8) + (T^2 + T + 1)(X^8Y^2+X^2Y^8) \\
&\hspace{-2.2cm} + (T^2 + T)(X^8Y+XY^8) + (T^6+T^5+T^3+T^2+1)(T^2+T)(X^8+Y^8)\\
&\hspace{-2.2cm} + (X^7Y^4+X^4Y^7)+ (T^2 + T)^3(X^7Y^3+X^3Y^7)\\
&\hspace{-2.2cm} + (T^5+T^4+T^3+T+1)(T^5+T^3+T^2+T+1)(T^2+T)^3(X^7+Y^7)\\
&\hspace{-2.2cm} + (X^6Y^5+X^5Y^6)+ (X^6Y^4+X^4Y^6) + (T^2+T+1)^5(X^6Y^3+X^3Y^6)\\
&\hspace{-2.2cm} + (T^7+T^6+T^5+T^4+T^2+T+1)(T^7+T^3+T^2+T+1)(T^2+T)(X^6Y^2+X^2Y^6)\\
&\hspace{-2.2cm} + (T^{14}+T^{13}+T^{11}+T^{10}+T^7+T^5+T^4+T^2+1)(T^2+T)^2(X^6Y+XY^6)\\
&\hspace{-2.2cm} + (T^4+T+1)(T^2+T+1)(T^2+T)^5(T^8+T^6+T^5+T^4+T^3+T+1)(X^6+Y^6)\\
&\hspace{-2.2cm} + X^5Y^5 + (T^2+T+1)(T^2+T)^2(X^5Y^4+X^4Y^5) + (T^2 + T)^2(X^5Y^3 +X^3Y^5)\\
&\hspace{-2.2cm} + (T^9+T^8+T^7+T^5+1)(T^9+T^7+T^6+T^3+T^2+T+1)(X^5Y^2+X^2Y^5)\\
&\hspace{-2.2cm} + (T^6+T^5+T^2+T+1)(T^6+T^5+1)(T^2+T+1)^3(T^2+T)^2(X^5Y+XY^5)\\
&\hspace{-2.2cm} + (T^5+T^3+T^2+T+1)(T^5+T^4+T^3+T+1)(T^2+T+1)(T^2+T)^5(X^5+Y^5)\\
&\hspace{-2.2cm} + (T^{18}+T^{17}+T^{16}+T^{10}+T^9+T^4+T^2+T+1)(T^2+T+1)^2(T^2+T)(X^4Y^2+X^2Y^4)\\
&\hspace{-2.2cm} + (T^2+T+1)^2(T^2+T)^7(X^4Y+XY^4)+ (T^2+T)^8(T^6+T^5+T^3+T^2+1)(X^4+Y^4)\\
&\hspace{-2.2cm} + (T^{10}+T^9+T^8+T^6+T^5+T+1)(T^2+T+1)^3X^3Y^3+(T^8+T^7+T^2+T+1)\\
&\hspace{-2.2cm} \cdot (T^8+T^7+T^6+T^5+T^4+T^3+1)(T^2+T+1)(T^2+T)^2(X^3Y^2+X^2Y^3)\\
&\hspace{-2.2cm} + (T^2+T+1)(T^2+T)^4(T^{10}+T^9+T^8+T^3+T^2+T+1)(X^3Y+XY^3)\\
&\hspace{-2.2cm} + (T^4+T+1)(T^3+T+1)(T^3+T^2+1)(T^2+T+1)^3(T^2+T)^3X^2Y^2\\
&\hspace{-2.2cm} + (T^2+T)^{10}(X^2Y+XY^2) + (T^2+T)^{10}(X^2+Y^2)+ (T^4+T+1)(T^2+T)^7(X^3+Y^3)\\
&\hspace{-2.2cm} + (T^3+T+1)(T^3+T^2+1)(T^2+T)^6XY + (T^2+T+1)(T^2+T)^8(X+Y)+ (T^2+T)^9\\
\end{split}
\end{equation*}
Finding a uniformizing element $u_0$ of $\F(X_0(T^2+T))$ and expressing $j_0$ and $j_1$ in it, we find
$$j_0=\dfrac{(u_0^3 + (T^2 + T)u_0 +(T^2 + T))^3}{u_0(u_0+T)^2(u_0+T+1)^2} \makebox{ and } j_1=\dfrac{(u_0^3 + (T^2 + T)u_0^2 +(T^2 + T)^2)^3}{u_0^4(u_0+T)^2(u_0+T+1)^2}$$
To find $f_{T^2+T}(X,Y)$, we need to factor a bivariate polynomial of $Y$-degree $9$. Note that Remark \ref{rem:casePprime} does not apply, though it still predicts the existence of one factor of $Y$-degree one. The factors turn out to be
$$ XY + T^2 + T,$$
$$ Y^2X^2 + TY^2X + (T^2 + T)YX + (T^3 + T^2)Y + T^2X^2 + T^4 + T^2,$$
$$ Y^2X^2 + (T + 1)Y^2X + (T^2 + T)YX + (T^3 + T)Y + (T^2 + 1)X^2 + T^4+ T^2,$$
and
\begin{equation*}
\begin{split}
Y^4X^3 + Y^4X^2 + (T^2 + T)Y^4X + (T^2 + T)Y^3X^2 + (T^2 + T)Y^3X +(T^4 + T^2)Y^3\\
+ (T^2 + T + 1)Y^2X^3 + (T^4 + T^2)Y^2X + (T^4 + T^2)Y^2 + (T^2 + T)YX^3\\
+ (T^4 + T)YX^2 + (T^6 + T^5 + T^4 + T^3)Y + X^4.
\end{split}
\end{equation*}
The last factor is $f_{T^2+T}(X,Y)$, since it is the only factor of $Y$-degree $4$.
Considering reduction modulo $T^2+T+1$, we see by Theorem \ref{thm:optimal} that the polynomial
$$Y^4X^3 + Y^4X^2 + Y^4X + Y^3X^2 + Y^3X +Y^3 + Y^2X + Y^2 + YX^3 + Y + X^4$$ recursively defines an optimal tower over $\mathbb{F}_{16}$.
\end{example}

\section{An example of a classical modular tower}

In \cite[Section 6.1.2.3]{loetter} a good recursive tower over the field $\F_{7^4}$ is given. The recursive equation stated there is: $$y^5=\frac{x^5+5x^4+x^3+2x^2+4x}{2x^4+5x^3+2x^2+x+1}.$$ We will consider the equivalent tower obtained by replacing $x$ by $3x$ and $y$ by $3y$. The resulting equation is:
\begin{equation}\label{eq:loetter}
y^5=x \frac{x^4-3x^3+4x^2-2x+1}{x^4+2x^3+4x^2+3x+1}
\end{equation}

The proof that the corresponding recursive tower is good can be carried out by observing that there are places that split completely in the tower and by observing that the ramification locus of the tower is finite. Since all ramification is tame (the steps in the tower are Kummer extensions), the Riemann-Hurwitz genus formula can be used directly to estimate the genera of the function fields occurring in the tower. In this way one obtains that the limit of the tower is at least $6$. The splitting places of this tower are not defined over $\F_{49}$, otherwise this would be an optimal tower. We will show in this section that this tower has a modular interpretation and obtain a generalization to other characteristics as well.

Based on the extension degrees, a reasonable supposition is that there may be a relation to the function fields of the curves $X_0(5^n)_{n\ge 1}$. In \cite{elkies} Elkies found an explicit recursive description of $X_0(5^n)_{n\ge 2}$: define $P(t):=t^5+5t^3+5t-11$, then this tower satisfies the recursive equation
$$P(y)=\dfrac{125}{P\left(\frac{x+4}{x-1}\right)},$$
or equivalently
\begin{equation}\label{eq:elkies}
y^5+5y^3+5y-11=\dfrac{(x-1)^5}{x^4+x^3+6x^2+6x+11}.
\end{equation}
The steps in this tower are not Galois, but Elkies notes that the polynomial $P(X)$ is dihedral. More concretely: $P(v^{-1}-v)=-v^5-11+v^{-5}.$ Since the steps in the recursive tower from equation \eqref{eq:loetter} are Galois (note that the $5$-th roots of unity belong to the constant field), we consider the extension $\Q(v)$ of $\Q(x)$ defined by $1/v-v=x$. Direct verification using MAGMA reveals that the function field $\Q(u,y)$ contains a solution $w$ to the equation $1/w-w=y$ such that $$w^5=v(v^4-3v^3+4v^2-2v+1)/(v^4+2v^3+4v^2+3v+1).$$ Therefore we recover equation \eqref{eq:loetter}. We have shown that the tower satisfying equation \eqref{eq:loetter} recursively, is a supertower of the modular tower $X_0(5^n)_{n \ge 2}.$ One can say more however. Equation \eqref{eq:loetter} occurs in the literature of modular functions. In fact it occurs in the same form in the famous first letter that S.~Ramanujan wrote $100$ years ago to G.H.~Hardy. In it, Ramanujan defined a continued fraction, now known as the Rogers--Ramanujan continued fraction, and related two of its values by equation \eqref{eq:loetter} (see Theorem 5.5 in \cite{berndt} for more details). The Rogers--Ramanujan continued fraction can be seen as a modular function for the full modular group $\Gamma(5)$ and defines a uniformizing element of the function field $\Q(X(5))$. This means that we can obtain the recursive tower defined (over $\Q$) by equation \eqref{eq:loetter} as a lift of the tower defined by equation \eqref{eq:elkies} by extending the first function field of that tower to the function field of $X(5)$. Also by direct computation one sees that the extension $\Q(\zeta_5)(w,x)/\Q(\zeta_5)(x)$ is a Galois extension (it is in fact the Galois closure of $\Q(\zeta_5)(x,y)/\Q(\zeta_5)(x)$).

For any prime number $p$ different from $5$ the curves have good reduction, meaning that we may reduce the equations modulo such primes $p$. Extending the constant field to $\F_{q}$ with $q=p^2$ if $p \equiv \pm 1 \pmod{5}$ and $q=p^4$ if $p \equiv \pm 2 \pmod{5}$, we make sure that the primitive fifth root of unity is contained in the constant field $\F_q$. Over this constant field, the tower satisfying the recursive relation \eqref{eq:loetter} has limit at least $p-1$, i.e., the ratio of the number of rational places and the genus tends to a value larger than or equal to $p-1$ as one goes up in the tower. This means that the tower is optimal if $p \equiv \pm 1 \pmod{5}$ and good if $p \equiv \pm 2 \pmod{5}$.

\section{A tower obtained from Drinfeld modules over a different ring}

Previously we have used Drinfeld modules of rank 2 over the ring $\F_q[T]$ to construct towers of function fields. In principle, one can consider Drinfeld modules over other rings and use them to construct towers of function fields. The theory is however, much less explicit in this case. In this section, we illustrate the method of constructing towers by studying a particular example in detail. More precisely, we consider Drinfeld modules over the ring $A:=\F_2[S,T]/\langle S^2 + S - T^3 - T \rangle.$  The ring $A$ is the coordinate ring of an elliptic curve with $5$ rational points. We denote by $P$ the prime ideal of $A$ generated by (the classes of) $S$ and $T$. This prime ideal corresponds to the point $(0,0)$ of the elliptic curve. We will construct an asymptotically good tower in this setup.

\subsection{Explicit Drinfeld modules of rank 2}

Unlike in the case of Drinfeld modules over the ring $\F_q[T]$ we cannot directly compute a modular polynomial. In fact, it is non-trivial even to compute examples of Drinfeld modules $\phi$ of rank two in this setting. Our first task will be to compute all possible normalized Drinfeld modules of rank $2$ over $A$ in characteristic $P$. Such a Drinfeld module $\phi$ is specified by
\begin{equation}\label{eq:phiT}
\phi_T=\tau^4+g_1\tau^3+g_2\tau^2+g_3\tau
\end{equation}
and
\begin{equation}\label{eq:phiS}
\phi_S=\tau^6+h_1\tau^5+h_2\tau^4+h_3\tau^3+h_4\tau^2+h_5\tau.
\end{equation}
The eight parameters $g_1,\dots,h_5$ cannot be chosen independently, but should be chosen such that $\phi_{S^2+S-T^3-T}=\phi_0=0$ and $\phi_T \phi_S=\phi_S \phi_T$. The first condition comes from the defining equation of the curve, while the second one should hold, since the fact the $\phi$ is a homomorphism implies that $\phi_T \phi_S=\phi_{TS}$ and $\phi_S \phi_T=\phi_{ST}=\phi_{TS}.$ In this way one obtains the following system of polynomial equations for $g_i$ and $h_j$.
From the condition $\phi_{S^2+S-T^3-T}=0$ one obtains that the $g_i$ and $h_i$ are in the zero-set of the following polynomials:

$$
\footnotesize
\begin{array}{l}
h_5 + g_3,\\
h_4 + h_5^3 + g_2,\\
h_3 + h_4^2h_5 + h_4h_5^4 + g_1 + g_3^7,\\
h_2 + h_3^2h_5 + h_3h_5^8 + h_4^5 + g_2^4g_3^3 + g_2^2g_3^9 + g_2g_3^{12} + 1,\\
h_1 + h_2^2h_5 + h_2h_5^{16} + h_3^4h_4 + h_3h_4^8 + g_1^4g_3^3 + g_1^2g_3^{17} +
        g_1g_3^{24} + g_2^{10}g_3 + g_2^9g_3^4 + g_2^5g_3^{16},\\
h_1^2h_5 + h_1h_5^{32} + h_2^4h_4 + h_2h_4^{16} + h_3^9 + g_1^8g_2^2g_3 + g_1^8g_2g_3^4
        + g_1^4g_2g_3^{32} + g_1^2g_2^{16}g_3 + g_1g_2^{16}g_3^8 + g_1g_2^8g_3^{32}\\
+ g_2^{21}       + g_3^{48} + g_3^{33} + g_3^3 + 1,\\
h_1^4h_4 + h_1h_4^{32} + h_2^8h_3 + h_2h_3^{16} + h_5^{64} + h_5 + g_1^{18}g_3 + g_1^{17}g_3^8
        + g_1^{16}g_2^5 + g_1^9g_3^{64} + g_1^4g_2^{33} + g_1g_2^{40} + g_2^{32}g_3^{16}\\
+      g_2^{32}g_3 + g_2^{16}g_3^{64} + g_2^2g_3 + g_2g_3^{64} + g_2g_3^4,\\
h_1^8h_3 + h_1h_3^{32} + h_2^{17} + h_4^{64} + h_4 + g_1^{36}g_2 + g_1^{33}g_2^8 +
        g_1^{32}g_3^{16} + g_1^{32}g_3 + g_1^{16}g_3^{128} + g_1^9g_2^{64} + g_1^2g_3 + g_1g_3^{128}\\
        + g_1g_3^8 + g_2^{80} + g_2^{65} + g_2^5,\\
h_1^{16}h_2 + h_1h_2^{32} + h_3^{64} + h_3 + g_1^{73} + g_1^{64}g_2^{16} + g_1^{64}g_2 +
        g_1^{16}g_2^{128} + g_1^4g_2 + g_1g_2^{128} + g_1g_2^8 + g_3^{256} + g_3^{16} + g_3,\\
h_1^{33} + h_2^{64} + h_2 + g_1^{144} + g_1^{129} + g_1^9 + g_2^{256} + g_2^{16} + g_2,\\
h_1^{64} + h_1 + g_1^{256} + g_1^{16} + g_1.
\end{array}$$
Similarly, the condition $\phi_T \phi_S=\phi_S \phi_T$ gives rise to the following polynomials:
$$
\footnotesize
\begin{array}{l}
    h_5^2g_3 + h_5g_3^2,\\
    h_4^2g_3 + h_4g_3^4 + h_5^4g_2 + h_5g_2^2,\\
    h_3^2g_3 + h_3g_3^8 + h_4^4g_2 + h_4g_2^4 + h_5^8g_1 + h_5g_1^2,\\
    h_2^2g_3 + h_2g_3^{16} + h_3^4g_2 + h_3g_2^8 + h_4^8g_1 + h_4g_1^4 + h_5^{16} + h_5,\\
    h_1^2g_3 + h_1g_3^{32} + h_2^4g_2 + h_2g_2^{16} + h_3^8g_1 + h_3g_1^8 + h_4^{16} + h_4,\\
    h_1^4g_2 + h_1g_2^{32} + h_2^8g_1 + h_2g_1^{16} + h_3^{16} + h_3 + g_3^{64} + g_3,\\
    h_1^8g_1 + h_1g_1^{32} + h_2^{16} + h_2 + g_2^{64} + g_2,\\
    h_1^{16} + h_1 + g_1^{64} + g_1.
\end{array}$$
One could attempt a direct Groebner basis computation on the ideal $I \subset \F_2[g_1,\dots,h_5]$ generated by the above two sets of polynomials, but we can simplify the system of polynomial equations first. Taking for example the last of each set of polynomials, $p_1:=h_1^{64} + h_1 + g_1^{256} + g_1^{16} + g_1$ and $p_2:=h_1^{16} + h_1 + g_1^{64} + g_1$, we find that $p_3:=p_1-p_2^4=h_1^4+h_1+g_1^{16}+g_1^4+g_1$ is an element of the ideal $I$. Moreover, since $p_2=p_3+p_3^4$ and $p_1=p_3+p_3^4+p_3^{16}$, we can replace $p_1$ and $p_2$ by $p_3$ when generating the ideal $I$. Also we can eliminate the variables $h_i$ altogether, since they can be expressed in terms of $g_1,g_2,g_3$ using the first five generators of $I$. After performing these and similar simplifications, we computed a Groebner basis of the resulting polynomial ideal in the variables $g_1,g_2$ and $g_3$ using MAGMA. The resulting Groebner basis contains one irreducible (but not absolutely irreducible) polynomial involving only $g_2$ and $g_3$ as well as an irreducible polynomial of degree one in $g_1$. This means that the zero-set of the ideal $I$ can be interpreted as an irreducible algebraic curve defined over $\F_2$. It turns out to have genus $4$.

From the modular point of view, it is more natural to consider isomorphism classes of Drinfeld modules. An isomorphism between two Drinfeld modules $\phi$ and $\psi$ is given by a non-zero constant $c$ such that $c\phi=\psi c$. Considering equations \eqref{eq:phiS} and \eqref{eq:phiT}, we see that for normalized Drinfeld modules $\phi$ and $\psi$ we have that $c \in \F_{4}$ and that $g_1^3,g_2,g_3^3,h_1^3,h_2,h_3^3,h_4,h_5^3$ are invariant under isomorphism. Inspecting the Groebner basis computation performed before, we obtain a polynomial relation between $g:=g_3^3$ and $g_2$ and a way to express all other invariants in these two parameters. These polynomials are too large to state here, so we will not do so. The important fact is that we again obtain an irreducible algebraic curve defined over $\F_2$ which determines the isomorphism classes of possible rank two Drinfeld modules. This modular curve is known to have genus zero and to be irreducible, but not absolutely irreducible, see \cite{gekeler}. There it is also shown that the number of components is equal to the class number $h_E$, over which extension field these components are defined and how the Galois group of this extension acts on the components. In our case we obtain that there are $5$ components defined over $\F_{32}$ and that the Frobenius map of $\F_{32}/\F_2$ acts transitively on these five components. One such component is determined by the following relation between $g$ and $g_2$:
$$
\scriptsize
\begin{array}{l}
g_2^{13} + (\alpha^5g + \alpha^{14})g_2^{12} + (\alpha^4g^2 + \alpha^{19}g +
        \alpha^7)g_2^{11} + (\alpha^9g^3 + \alpha^{18}g^2 + \alpha^9g +
        \alpha^{21})g_2^{10}\\
+ (\alpha^{10}g^4 + \alpha^{21}g^3 + \alpha^{16}g^2 +
        \alpha^{18}g + \alpha^8)g_2^9 + (\alpha^{15}g^5 + \alpha^{29}g^4 + \alpha^{10}g^3
        + \alpha^{27}g^2 + \alpha^{25}g + \alpha^8)g_2^8\\
+ (g^6 + \alpha^{28}g^5 +
        \alpha^6g^4 + \alpha^{11}g^3 + \alpha^6g^2 + \alpha^{28}g + \alpha^9)g_2^7\\
+ (\alpha^5g^7 + \alpha^{23}g^6 + \alpha^2g^5 + \alpha^{15}g^4 + \alpha^{12}g^3
        + \alpha^4g^2 + \alpha^6g + \alpha^{25})g_2^6\\
+ (\alpha^4g^8 + \alpha^{30}g^7
        + \alpha^{18}g^6 + \alpha^3g^5 + \alpha^{15}g^4 + \alpha^{12}g^3 +
        \alpha^{23}g^2 + \alpha^{29}g + \alpha^{10})g_2^5\\
+ (\alpha^9g^9 + \alpha^{25}g^8
        + \alpha^8g^7 + \alpha g^6 + \alpha^7g^5 + \alpha^{25}g^4 + \alpha^{23}g^3 +
        \alpha^{15}g^2 + \alpha g + \alpha^{26})g_2^4\\
+ (\alpha^4g^{10} + \alpha^{27}g^9 +
        \alpha^{15}g^8 + \alpha^{11}g^7 + \alpha^5g^6 + \alpha^{26}g^5 + \alpha^{18}g^4
        + \alpha^9g^3 + \alpha^{11}g^2 + \alpha^{30}g)g_2^3\\
+ (\alpha^9g^{11} +
        \alpha^{30}g^{10} + \alpha^{10}g^9 + \alpha^{15}g^8 + \alpha^{12}g^7 + \alpha^6g^6
        + \alpha^2g^5 + \alpha^{26}g^4 + \alpha^{15}g^3 + \alpha^6g^2 + \alpha^{13}g +
        \alpha^{30})g_2^2\\
+ (\alpha^{10}g^{12} + \alpha^{16}g^{11} + \alpha^4g^{10} +
        \alpha^{12}g^9 + \alpha^{18}g^8 + \alpha^{28}g^7 + \alpha^2g^6 + \alpha^9g^5 +
        \alpha^3g^4 + \alpha^8g^3 + \alpha^{10}g^2 + \alpha^{17}g)g_2\\
+ \alpha^{15}g^{13} + \alpha^5g^{12} + \alpha^{24}g^{11} + \alpha^4g^{10} +
        \alpha^{11}g^9 + \alpha^8g^8 + \alpha^{12}g^7 + \alpha^{27}g^6 + g^5 +
        \alpha^{23}g^4 + \alpha^{19}g^3 + \alpha^8g^2\\ + \alpha^{24}g + 1,
\end{array}
$$
with $\alpha^5+\alpha^2+1=0.$

 Using this polynomial, we can define a rational function field $\F_{32}(g_2,g)$. Since it is rational, there exists a uniformizer $u \in \F_{32}(g_2,g)$ such that $\F_2(g_2,g)=\F_2(u)$. Finding such element $u$ can easily be done using MAGMA. Note that this element $u$ plays a very similar role as the element $j_0$ in Section \ref{section:two}, since it describes isomorphism classes of rank two Drinfeld modules. The only difference is that now there exist five conjugated families of isomorphism classes, whereas previously there was only one such family.

\subsection{Finding an isogeny}

 To find a tower, we need to find an isogeny from a given Drinfeld module to another. That is to say: we need to find two Drinfeld modules $\phi$ and $\psi$ both of rank two and an additive polynomial $\lambda$ such that $\lambda \phi = \psi \lambda$. We will describe the most direct approach, not using the theory of torsion points, which would give a faster way to obtain isogenies. We will find an isogeny $\lambda$ of the simplest possible form $\lambda=\tau - a$ from $\phi$ to another Drinfeld module $\psi$ specified by
\[
\psi_T:=\tau^4+l_1\tau^3+l_2\tau^2+l_3\tau
\]
and
\[
\psi_S=\tau^6 +t_1\tau^5 +t_2\tau^4 +t_3\tau^3 +t_4\tau^2 +t_5\tau.
\]
Since we can describe both $\phi$ and $\psi$ essentially using only one parameter, we can obtain a relation between these parameters and $a$. More in detail, always assuming $q=2$, we have
\begin{equation}\label{eq:T}
\lambda \phi_T = \psi_T \lambda
\end{equation} and
\begin{equation}\label{eq:S}
\lambda \phi_S = \psi_S \lambda
\end{equation}
The left hand side of equation \eqref{eq:T} is
\begin{align*}
& (\tau -a)(\tau^4+g_1\tau^3+g_2\tau^2+g_3\tau) \\
&= \tau^5 +(g_1^q-a)\tau^4 + (g_2^q-ag_1)\tau^3 + (g_3^q -ag_2)\tau^2 -ag_3\tau
\end{align*}
while the right hand one is
\begin{align*}
& (\tau^4+l_1\tau^3+l_2\tau^2+l_3\tau)(\tau -a) \\
&= \tau^5 +(l_1-a^{q^4})\tau^4 + (l_2-l_1a^{q^3})\tau^3 + (l_3 -l_2a^{q^2})\tau^2 -l_3a^q\tau.
\end{align*}
Consequently we get
\begin{equation}
\begin{cases}
g_1^q-a &= l_1-a^{q^4} \\
g_2^q-ag_1 &= l_2-l_1a^{q^3} \\
g_3^q -ag_2 &= l_3 -l_2a^{q^2} \\
-ag_3 &= -l_3a^q \label{eq:agl}
\end{cases}
\end{equation}
By substitution top down, we can eliminate variables $l_1, l_2, l_3$ and get
\begin{equation*}
(g_1a^{q^2+q+1} + g_2a^{q+1} + g_3a + a^{q^3+q^2+q+1})^q - (g_1a^{q^2+q+1} + g_2a^{q+1} + g_3a + a^{q^3+q^2+q+1})=0
\end{equation*}
or
\begin{equation}\label{eq:Tg}
a^{q^3+q^2+q+1} + g_1a^{q^2+q+1} + g_2a^{q+1} + g_3a = \gamma \in \F_q
\end{equation}
Equation \eqref{eq:Tg} can be seen as a polynomial in terms of $a, u$ and $g_3$.


Similarly, studying equation \eqref{eq:S}, we obtain
\begin{equation}
\begin{cases}
h_1^q-a &= t_1-a^{q^6} \\
h_2^q-ah_1 &= t_2-t_1a^{q^5} \\
h_3^q-ah_2 &= t_3-t_2a^{q^4} \\
h_4^q-ah_3 &= t_4-t_3a^{q^3} \\
h_5^q-ah_4 &= t_5-t_4a^{q^2} \\
-ah_5 &= -t_5a^q \label{eq:aht}
\end{cases}
\end{equation}
Also by substitution, we can eliminate variables $t_i (i=1,\dots,5)$ and obtain similarly
\begin{equation}\label{eq:Sh}
a^{q^5+q^4+q^3+q^2+q+1} + h_1a^{q^4+q^3+q^2+q+1} + h_2a^{q^3+q^2+q+1} + h_3a^{q^2+q+1} + h_4a^{q+1} + h_5a = \beta
\end{equation}
with $\beta \in \F_q$. As $h_i (i=1,\dots,5)$ can be expressed in terms of $g_1, g_2$ and $g_3$, the equation \eqref{eq:Sh} can be seen as a polynomial in $a,u$ and $g_3$ as well. Choosing $\beta=\gamma=1$ and computing the greatest common divisor of the resulting polynomials in equations \eqref{eq:Tg} and \eqref{eq:Sh} gives rise to an algebraic condition on $a$ of degree three. As an aside, note that the choice $\beta=\gamma=1$ corresponds to finding a $\langle S+1,T+1 \rangle$-isogeny. We obtain that the Drinfeld module $\psi$ can be expressed in terms of $u,g_3$ and $a$. Now recall that $l_1^3, l_2$ and $l_3^3$ can also be expressed in some $v \in \F_{32}(l_2,l_3^3)$. It turns out that $\psi$ does not correspond to a point in the same family of $\phi$, but a conjugated one. In this case we need to apply Frobenius three times to go from the family to which the isomorphism class of $\phi$ belongs, to the family to which the isomorphism class of $\psi$ belongs. Relating the parameters $u$ and $v$ we obtain that $\Phi(\alpha,u,v)=0$ with
$$\Phi(\alpha,X,Y) := (X^3+\alpha^{24}X^2+\alpha^4X+\alpha^9)Y^3+(\alpha^{17}X^3+\alpha^{29}X^2+X+\alpha^{30})Y^2$$
\begin{equation}\label{eq:vv}
 + (\alpha^{30}X^3+\alpha^{12}X^2+\alpha^{30}X+\alpha^{17})Y+(\alpha^4X^3+\alpha^{14}X^2+\alpha^{19}).
\end{equation}
As noted before, the parameter $u$ plays the same role as $j_0$ from Section \ref{section:two}. Similarly $v$ plays the same role as $j_1$ and the polynomial $\Phi(\alpha,X,Y)$ can be seen as an analogue of a Drinfeld modular polynomial $\Phi_N(X,Y)$. For completeness, let us note that whereas $N$ was a polynomial before, its role is now taken by the ideal $\langle S+1,T+1 \rangle \subset A$ which implicitly played a role in the construction of the isogeny $\lambda$.

\subsection{Obtaining a tower}

Just as for the towers from Section \ref{section:two}, we need a quadratic extension of the constant field in order to obtain many rational places. From now on we will therefore work over the field $\F_{2^{10}}$ instead of $\F_{2^{5}}$. Let $\beta \in \F_{2^{10}}$ be a primitive element, the $\alpha$'s of the polynomial \eqref{eq:vv} should be changed in terms of $\beta$ using the relation $\alpha=\beta^{33}$. We would now like to define a tower $\mathcal{F}:=(F_0 \subset F_1 \subset \cdots)$ of function fields as follows:
\begin{equation}\label{eq:towdef}
F_0:=\F_{2^{10}}(u_0) \, \makebox{ and for $n  \ge 0$ } \, F_{n+1}:=F_{n}(u_{n+1}),
\end{equation}
with $\Phi(\alpha^{8^{n}},u_{n},u_{n+1})=0.$ There are two remarks to be made. In the first place, the reason one needs to take $\alpha^{2^n}$ as argument is that in the first iteration we went from one family of rank two Drinfeld modules to another (namely the one obtained by applying Frobenius three times). In the next iteration one therefore needs to start at this family. This amounts to replacing $\alpha$ by $\alpha^8$ in equation \eqref{eq:vv}. Iteratively in the $n+1$-th step we need to replace $\alpha$ by $\alpha^{8^n}$. The second remark is that in fact the polynomial $\Phi(\alpha^{8},u_{1},T) \in F_1[T]$ is not irreducible. It has the degree one factor $(u_0 + \alpha^{25})T + (\alpha^{28}u_0 + \alpha^{27})
$ and a degree two factor. This is in perfect analogy with Proposition \ref{prop:depth two}. To define the tower more accurately, we would have to specify this degree two factor and use that to define $F_n$ if $n>1$. A direct computation reveals there is always a totally ramified place with ramification index two in the extension $F_{n+1}/F_n$ for $n>0$ and hence that the degree two factor remains irreducible. This means that all the steps in the tower, except the first one, are Artin--Schreier extensions.

A careful analysis of the extension $F_1/F_0$ reveals the following:
\begin{proposition}\label{prop:basicfunctfield}
The extension $F_1/F_0$ satisfies the following:
\begin{enumerate}
\item $[F_1:F_0]=3$,
\item The place $[u_0=\beta^{858}]$ is totally ramified, i.e., it has ramification index $3$.
\item The places $[u_0=\beta^{165}], [u_0=\beta^{368}],[u_0=\beta^{523}],$ and $[u_0=\beta^{891}]$ are completely splitting.
\item Above each of the places $[u_0=\beta^{198}], [u_0=\beta^{330}],[u_0=\beta^{528}],[u_0=\beta^{627}],$ and $[u_0=\beta^{924}]$ lie two places of $F_1$. One of these two has ramification index $2$ and different exponent $2$, the other has ramification index one.
\item The genus of $F_1$ is $4$.
\end{enumerate}
\end{proposition}
\begin{proof}
All this follows by a direct computation, for example using MAGMA.
\end{proof}
The place mentioned, though ramified in the first extension turns out to split completely in all subsequent extensions. More precisely, denote by $P$ the place of $F_1$ lying above $[u_0=\beta^{858}]$. Then one can show that $P$ splits completely in any of the extensions $F_n/F_1$ for $n>1$. Using the recursive structure of the tower $\mathcal F$, it is not hard to show this. Combining this with part (iii) of the above proposition, this yields the following:
\begin{lemma}\label{lem:splitting}
Let $n>0$. The number of rational places of $F_n$ is at least $13\cdot 2^{n-1}$.
\end{lemma}
Also the genus of the function fields in the tower $\mathcal F$ can be estimated. Recall that $F_{n+1}/F_n$ is an Artin--Schreier extension if $n>0$. Using the recursive nature of the tower and either direct computation or a computer program like MAGMA, one can show that all ramification in the extension $F_2/F_1$ is $2$-bounded, that is that for any place $P$ of $F_1$ and any place $Q$ of $F_2$ lying above $F_1$, we have $d(Q|P)=2e(Q|P)-2$. The same is true for the extension $F_2/\F_{2^{10}}(u_1,u_2)$. By \cite[Lemma 1]{GSartin} and the recursive definition of the tower, this means that for any $n>1$, the ramification in the extension $F_n/F_1$ is $2$-bounded. By part (iv) of Proposition \ref{prop:basicfunctfield}, there are exactly $10$ places of $F_1$ that may ramify in $F_n/F_1$. Using Riemann--Hurwitz and the $2$-boundedness of the ramification, we obtain for any $n>1$ that
\begin{align*}
2g(F_n) -2 &= 2^{n-1}(2\cdot 4  - 2) +\deg \mathrm{Diff}(F_n/F_1) \\
  &\leq 2^{n-1}6 + 10\cdot 2\cdot 2^{n-1}.
\end{align*}
Hence we obtain the following:
\begin{lemma}\label{lem:genus}
For $n>1$ we have $g(F_n) \leq 13\cdot 2^{n-1} + 1.$
\end{lemma}
This shows that the tower $\mathcal F$ is good. More precisely, we obtain from Lemmas \ref{lem:splitting} and \ref{lem:genus} that:
$$\lambda(\mathcal F) \ge 1.$$ In other words, the tower defined by equation \eqref{eq:towdef} is asymptotically good.

\bibliographystyle{99}

\vspace{2cm}

\noindent Alp Bassa\\
Sabanc{\i} University\\
{\rm 34956} Tuzla, \.Istanbul, Turkey\\
bassa@sabanciuniv.edu\\

\noindent Peter Beelen\\
Technical University of Denmark\\
Matematiktorvet, Building 303B\\
DK-2800, Lyngby, Denmark\\
pabe@dtu.dk\\

\noindent Nhut Nguyen\\
Technical University of Denmark\\
Matematiktorvet, Building 303B\\
DK-2800, Lyngby, Denmark\\
nhngu@dtu.dk\\
\end{document}